\documentclass[a4paper,11pt]{amsart}

\usepackage{mathrsfs}
\usepackage{color}
\usepackage{multirow,bigdelim}
\usepackage[all]{xy}
\usepackage[driverfallback=dvipdfm]{hyperref}
\usepackage{geometry}
\newtheorem{thm}{Theorem}[section]
\newtheorem{lem}[thm]{Lemma}
\newtheorem{cor}[thm]{Corollary}

\newtheorem{claim}[thm]{Claim}

\theoremstyle{definition}
\newtheorem{defin}[thm]{Definition}
\newtheorem{eg}[thm]{Example}
\newtheorem*{conventions}{Conventions}
\newtheorem*{notation}{Notation}

\theoremstyle{remark}
\newtheorem{rem}[thm]{Remark}

\numberwithin{equation}{section}

\newcommand{\bA}{\mathbb{A}}

\newcommand{\fC}{\mathfrak{C}}

\newcommand{\fD}{\mathfrak{D}}
\newcommand{\bF}{\mathbb{F}}

\newcommand{\kc}{\overline{k}}
\newcommand{\sL}{\mathscr{L}}

\newcommand{\sO}{\mathscr{O}}
\newcommand{\bP}{\mathbb{P}}

\newcommand{\bQ}{\mathbb{Q}}

\newcommand{\bZ}{\mathbb{Z}}

\newcommand{\Bs}{\mathrm{Bs}}

\newcommand{\Gal}{\mathrm{Gal}(\overline{k}/k)}

\newcommand{\Pic}{\mathrm{Pic}}

\newcommand{\Spec}{\mathrm{Spec}}

\begin{document}

\title{Notes on cylinders in smooth projective surfaces}

%    Remove any unused author tags.

%    author one information
\author{}
\address{}
\curraddr{}
\email{}
\thanks{}

%    author two information
\author{Masatomo Sawahara}
\address{Graduate School of Science and Engineering, Saitama University, Shimo-Okubo 255, Sakura-ku Saitama-shi,  Saitama 338-8570, JAPAN}
\curraddr{}
%\email{sawaharam@mail.saitama-u.ac.jp}
\email{sawahara.masatomo@gmail.com}
\thanks{}



\date{}

\dedicatory{}

\begin{abstract}
In this article, we determine the existing condition of cylinders in smooth minimal geometrically rational surfaces over a perfect field. 
Furthermore, we show that for any birational map between smooth projective surfaces, one contains a cylinder if and only if so does the other. 
\end{abstract}

\maketitle
\setcounter{tocdepth}{1}

%%%%%%%%%%%%%%%%%
Throughout this article, let $k$ be a perfect field of an arbitrary characteristic and let $\kc$ be an algebraic closure of $k$. 
%%%%%%%%%%%%%%%%%

%%%%%%%%%%%%%%%%%%%%%%%%%%%%%%%%%%%%%%%%%%%%%%%%%%%%%%%%%%%%%%%%%%%%%%%%%%%%%%%%%%%%%%%%%%%%%%%%%%%%%%%%%%
\section{Introduction}\label{1}
An open subset $U$ of an algebraic variety $X$ over $k$ is called an {\it $\bA ^r_k$-cylinder} if $U$ is isomorphic to $\bA ^r_k \times Z$ for some variety $Z$ over $k$. 
Certainly, cylinders are geometrically simple objects, however, they receive a lot of attention from the viewpoint of unipotent group actions on affine cones over polarized varieties (e.g., {\cite{KPZ11,Per13,CPW16a}}). 
Recently, the importance of finding ``relative'' cylinders with respect to Mori fiber spaces is particularly recognized ({\cite[\S 3.3]{CPPZ21}}), where a relative cylinder means a {\it vertical cylinder} defined by: 
%%%%%%%%%%%%%%%%%
\begin{defin}[{\cite{DK18}}]\label{vertical:def}
Let $f:X \to Y$ be a dominant projective morphism of relative dimension $s \ge 1$ defined over a field $K$, and let $r$ be an integer with $1 \le r \le s$. 
For an integer $r$ with $1 \le r \le s$, an $\bA ^r_K$-cylinder $U \simeq \bA ^r_K \times Z$ in $X$ is called a {\it vertical $\bA ^r_K$-cylinder} with respect to $f$ if there exists a morphism $g: Z \to Y$ (of relative dimension $s-r$) such that the restriction of $f$ to $U$ coincides with $g \circ pr _Z$. 
\end{defin}
%%%%%%%%%%%%%%%%%
Moreover, we note the following fact: 
%%%%%%%%%%%%%%%%%
\begin{lem}[{\cite[Lemma 3]{DK18}}]\label{vertical}
Let the notation be the same as in Definition \ref{vertical:def}. 
Then $f$ admits a vertical $\bA ^r_K$-cylinder if and only if the generic fiber $X_{\eta}$, which is defined over the function field $K(Y) = K(\eta )$ of the base variety, contains an $\bA ^r_{K(Y)}$-cylinder. 
\end{lem}
%%%%%%%%%%%%%%%%%
Lemma \ref{vertical} implies that looking for vertical cylinders with respect to a Mori fiber space can be attributed to the determination of the existence of cylinders in the generic fiber, which is a Fano variety of Picard rank one, of this. 
Here, letting $f:X \to Y$ be a dominant morphism between normal algebraic varieties defined over a field $K$ satisfying $\dim Y >0$, we note that the function field $K(Y)$, which is the base field of the generic fiber of $f$, is not algebraically closed even if $K$ is algebraically closed. 
Hence, we need to treat algebraic varieties defined over an algebraically non-closed field. 

Now, we shall focus on the case of a Mori fiber space of relative dimension two, more precisely a so-called del Pezzo fibration. 
In fact, Dubouloz and Kishimoto completely provided the existence condition of cylinders in smooth del Pezzo surfaces of Picard rank one (see Theorem \ref{DK} (1)). 
Hence, we can determine whether a del Pezzo fibration (with only $\bQ$-factorial terminal singularities) admits a vertical cylinder or not. 
In this article, we generalize their consideration to the case of smooth projective surfaces instead of smooth del Pezzo surfaces of Picard rank one. 
In other words, our target is thus the existing condition of cylinders in smooth projective surfaces defined over an algebraically non-closed field. 
%%%%%%%%%%%%%%%%%

We shall first consider the case of minimal del Pezzo surfaces. 
Note that the Picard rank of a smooth minimal del Pezzo surface is either one or two (see Lemma \ref{minimal}). 
As mentioned above, the case of smooth del Pezzo surfaces of Picard rank one was studied by Dubouloz and Kishimoto. 
Moreover, the author recently gives the existing condition of cylinders in smooth minimal del Pezzo surfaces of Picard rank two. 
These results are summarized as the following theorem: 
%%%%%%%%%%%%%%%%%
\begin{thm}[{\cite[Theorem 1]{DK18}}, {\cite[Theorem 1.7]{Saw1}}]\label{DK}
Let $V$ be a smooth minimal del Pezzo surface defined over a perfect field $k$. 
Then the following assertions hold: 
\begin{enumerate}
\item Assume that $V$ is of Picard rank one. 
Then $V$ contains an $\bA ^1_k$-cylinder if and only if $(-K_V)^2 \ge 5$ and $V(k) \not= \emptyset$. 
\item Assume that $V$ is of Picard rank two. 
Then $V$ contains an $\bA ^1_k$-cylinder if and only if $(-K_V)^2 =8$ and $V$ is endowed with a structure of Mori conic bundle admitting a section defined over $k$. 
\end{enumerate}
\end{thm}
%%%%%%%%%%%%%%%%%
\begin{rem}
In Theorem \ref{DK} (1), it is also known that $V$ contains an $\bA ^1_k$-cylinder if and only if $V$ is rational over $k$ (see {\cite[Theorem 3.30]{CPPZ21}}). 
However, note that a smooth minimal del Pezzo surface of Picard rank two containing an $\bA ^1_k$-cylinder is not always rational over $k$ by Theorem \ref{DK} (2). 
\end{rem}
%%%%%%%%%%%%%%%%%
{\cite{DK18}} and {\cite{Saw1}} only deal with the case of characteristic zero, but note that Theorem \ref{DK} holds for an arbitrary characteristic. 
Indeed, their arguments work verbatim even if ${\rm char}(k)>0$. 

Next, we shall consider the case of minimal projective surfaces but not del Pezzo surfaces. 
It is enough to treat only the case of minimal geometrically rational because the other cases are not difficult (see Remark \ref{rem}). 
{\cite{Saw1}} also gives the existence condition of cylinders in minimal weak del Pezzo surfaces with its anti-canonical divisor not ample. 
By generalizing this work to minimal geometrically rational projective surfaces over perfect fields, we have the first main result as follows: 
%%%%%%%%%%%%%%%%%
\begin{thm}\label{main(1)}
Let $V$ be a smooth minimal geometrically rational projective surface, whose $-K_V$ is not ample, defined over a perfect field $k$. 
Then the following assertions hold: 
\begin{enumerate}
\item $V$ contains an $\bA ^1_k$-cylinder if and only if $(-K_V)^2=8$. 
\item $V$ contains the affine plane $\bA ^2_k$ if and only if $(-K_V)^2=8$ and $V(k) \not= \emptyset$. 
\end{enumerate}
\end{thm}
%%%%%%%%%%%%%%%%%
\begin{rem}\label{rem}
Let $V$ be a smooth minimal projective surface over $k$. 
We note $\kappa (V_{\kc}) = -\infty$ provided that $V$ contains an $\bA ^1_k$-cylinder. 
Indeed, $V_{\kc}$ is then birational to a smooth projective surface with a $\bP ^1$-fibration structure. 
Now, we further assume that $V_{\kc}$ is ruled but irrational. 
Notice that $V$ is endowed with only one structure of Mori conic bundle $\pi :V \to B$ over $k$. 
Then any $\bA ^1_k$-cylinder on $V$ is a vertical $\bA ^1_k$-cylinder with respect to $\pi$ (cf.\ Lemma \ref{lem(5-2)}). 
Thus, $V$ contains an $\bA ^1_k$-cylinder if and only if $\pi$ admits a section defined over $k$ by Lemma \ref{vertical}. 
\end{rem}
%%%%%%%%%%%%%%%%%
From now on, we shall consider the case of not necessarily minimal projective surfaces. 
Then we mainly focus on the classification of elementary links between smooth geometrically rational projective surfaces over perfect fields (see {\cite{Isk96}}, for details). 
As a result, we obtain the second main result as follows: 
%%%%%%%%%%%%%%%%%
\begin{thm}\label{main(2)}
Let $\sigma :S \to V$ be a birational morphism between smooth projective surfaces over a perfect field $k$ such that $V$ is $k$-minimal. 
Then $S$ contains an $\bA ^1_k$-cylinder if and only if so does $V$. 
\end{thm}
%%%%%%%%%%%%%%%%%
By Theorems \ref{DK}, \ref{main(1)} and \ref{main(2)}, letting $\sigma :S \to V$ be a birational morphism between smooth geometrically rational projective surfaces defined over $k$, we see that $S$ contains an $\bA ^1_k$-cylinder if and only if $V$ satisfies one of the following: 
%%%%%%%%%%%%%%%%%
\begin{itemize}
\item $V$ is a del Pezzo surface of Picard rank one with $(-K_V)^2 \ge 5$ and $V(k) \not= \emptyset$; 
\item $V$ is a $k$-form of the Hirzebruch surface $\bF _m$ of degree $m \ge 2$; 
\item $V$ is a $k$-form of $\bP ^1_{\kc} \times \bP ^1_{\kc}$ and is endowed with a structure of Mori conic bundle admitting a section defined over $k$. 
\end{itemize}
%%%%%%%%%%%%%%%%%

On the other hand, Theorem \ref{main(2)} immediately provides the following corollary: 
%%%%%%%%%%%%%%%%%
\begin{cor}\label{cor}
Let $\chi :S_1 \dashrightarrow S_2$ be a birational map between smooth projective surfaces over a perfect field $k$. 
Then $S_1$ contains an $\bA ^1_k$-cylinder if and only if so does $S_2$. 
\end{cor}
%%%%%%%%%%%%%%%%%
The article is organized as follows. 
In Section \ref{2}, we review some facts on smooth minimal geometrically ruled projective surfaces over an algebraically non-closed field and the Mori conic bundles from these surfaces. 
In section \ref{3}, we recall that birational morphism between smooth minimal geometrically rational surfaces over a perfect field (of an arbitrary characteristic), more precisely the classification of elementary links between these surfaces. 
The results in this section will play important role in Section \ref{5}. 
In section \ref{4}, we prove Theorem \ref{main(1)}. 
The idea of this proof is mainly based on the author's previous work {\cite{Saw1}}. 
In section \ref{5}, we prove Theorem \ref{main(2)}. 
Then we easily see that it is enough to show the ``only if'' part of Theorem \ref{main(2)}. 
In other words, for any birational morphism $\sigma :S \to V$ between smooth projective surfaces over $k$, whose $V$ is $k$-minimal, we will show that $V$ contains an $\bA ^1_k$-cylinder provided that $S$ contains an $\bA ^1_k$-cylinder. 
In the last section \ref{6}, we present some remarks on Corollary \ref{cor}. 
%%%%%%%%%%%%%%%%%%%%%%%%%%%%%%%%%%%%%%%%%%%%%%%%%%%%%%%%
\begin{conventions}
Let $X$ be an algebraic variety defined over $k$. 
Then $X_{\kc}$ denotes the base extension of $X$ to the algebraic closure $\kc$, i.e., $X_{\kc} := X \times _{\Spec (k)} \Spec (\kc )$. 
$X$ is {\it geometrically rational} (resp. {\it geometrically ruled}) if $X_{\kc}$ is rational (resp. ruled). 
Let $\pi :X \to Y$ be a surjective morphism between smooth projective varieties defined over $k$. 
Then we say that $\pi$ is a {\it $\bP ^1$-fibration} (resp. {\it a $\bP ^1$-bundle}) if a general fiber (resp. any fiber) of the base extension $\pi_{\kc}: X_{\kc} \to Y_{\kc}$ is isomorphic to $\bP ^1_{\kc}$. 
Furthermore, we say that $\pi$ is a {\it Mori conic bundle} if any fiber of the base extension $\pi_{\kc}: X_{\kc} \to Y_{\kc}$ is isomorphic to the plane conic (not necessarily irreducible). 
A normal projective surface $S$ is a {\it del Pezzo surface} if the anti-canonical divisor $-K_S$ is ample. 
A smooth projective surface $V$ defined over $k$ is a {\it $k$-minimal} if any birational morphism $\mu :V \to V'$ from $V$ to a smooth projective surface $V'$ defined over $k$ is an isomorphism. 
\end{conventions}
%%%%%%%%%%%%%%%%%%%%%%%%%%%%%%%%%%%%%%%%%%%%%%%%%%%%%%%%
\begin{notation}
We will use the following notations: 
\begin{itemize}
\item $\kappa (V)$: the Kodaira dimension of a smooth projective surface $V$ over an algebraically closed field. 
\item $q(V)$: the irregularity number of a smooth projective surface $V$ over an algebraically closed field. 
\item $\rho _k(X)$: the Picard number of a variety $X$ defined over $k$. 
\item $\Pic (X)$: the Picard group of a variety $X$. 
\item $\Pic (X)_{\bQ} := \Pic (X) \otimes _{\bZ} \bQ$. 
\item $\varphi ^{\ast}(D)$: the total transform of a divisor $D$ by a morphism $\varphi$. 
\item $\psi _{\ast}(D)$: the direct image of a divisor $D$ by a morphism $\psi$. 
\item $(D \cdot D')$: the intersection number of two divisors $D$ and $D'$. 
\item $(D)^2$: the self-intersection number of a divisor $D$. 
\item $\bF _m$: the Hirzebruch surface of degree $m$, i.e., $\bF _m = \bP (\sO _{\bP ^1} \oplus \sO _{\bP ^1}(m))$. 
\end{itemize}
\end{notation}
%%%%%%%%%%%%%%%%%%%%%%%%%%%%%%%%%%%%%%%%%%%%%%%%%%%%%%%%%%%%%%%%%%%%%%%%%%%%%%%%%%%%%%%%%%%%%%%%%%%%%%%%%%
\section{Preliminaries}\label{2}
In this section, we summarize some basic but important facts for this article. 
%%%%%%%%%%%%%%%%%%%
\begin{lem}\label{BS}
Let $X$ be a smooth algebraic variety over $k$ satisfying $X_{\kc} \simeq \bP ^n_{\kc}$. 
Then $X \simeq \bP ^n_k$ if and only if $X(k) \not= \emptyset$. 
\end{lem}
%%%%%%%%%%%%%%%%%%%
\begin{proof}
See, e.g., {\cite[Proposition 4.5.10]{Poo17}}. 
\end{proof}
%%%%%%%%%%%%%%%%%%%
In this article, we will deal with smooth minimal geometrically ruled surfaces over an algebraically non-closed field. 
We then note the following fact: 
%%%%%%%%%%%%%%%%%%%
\begin{lem}\label{minimal}
Let $V$ be a smooth minimal geometrically ruled surface over $k$. 
Then either $V$ is a del Pezzo surface of Picard rank one or $V$ is of Picard rank two and is endowed with a structure of Mori conic bundle defined over $k$. 
\end{lem}
%%%%%%%%%%%%%%%%%%%
\begin{proof}
Note that $K_V$ is not nef since $V$ is geometrically ruled. 
Hence, we obtain the assertion by {\cite[Theorem 9.3.20]{Poo17}}. 
\end{proof} 
%%%%%%%%%%%%%%%%%%%
Hence, we shall mainly consider the properties of Mori conic bundles from smooth minimal geometrically ruled surfaces. 
%%%%%%%%%%%%%%%%%%%
Note that {\cite[Exercise 3.13]{KSC04}} summarizes the basic properties of Mori conic bundles from a smooth projective surface. 
In this article, we will particularly use the following properties: 
%%%%%%%%%%%%%%%%%%%
\begin{lem}\label{MCB}
Let $V$ be a smooth projective surface over $k$ with a Mori conic bundle structure $\pi :V \to B$. Then: 
\begin{enumerate}
\item The number of singular fibers of $\pi _{\kc}$ equals $8(1-q(V_{\kc}))-(-K_V)^2$. 
In particular, $\pi$ is a $\bP ^1$-bundle if and only if $(-K_V)^2=8(1-q(V_{\kc}))$. 
\item Assume further that $V$ is minimal. Then $\pi$ admits a section defined over $k$ only if $\pi$ is a $\bP ^1$-bundle, in other words, $(-K_V)^2=8(1-q(V_{\kc}))$. 
\end{enumerate}
\end{lem}
%%%%%%%%%%%%%%%%%%%
\begin{proof}
By using {\cite[Exercise 3.13]{KSC04}}, we obtain these assertions (see also {\cite[Lemma 2.4]{Saw1}}). 
\end{proof}
%%%%%%%%%%%%%%%%%%%
On the other hand, we prepare the following lemma: 
%%%%%%%%%%%%%%%%%%%
\begin{lem}\label{degree}
Let $V$ be a smooth minimal geometrically rational projective surface of Picard rank two over $k$. 
If $(-K_V)^2 \not= 8$, then $(-K_V)^2 \le 4$.  
\end{lem}
%%%%%%%%%%%%%%%%%%%
\begin{proof}
This assertion follows from {\cite[Theorems 5]{Isk80}} and {\cite[Appendix A]{Saw1}}. 
\end{proof}
%%%%%%%%%%%%%%%%%%%
Finally, we present the following result, which generalizes {\cite[Lemma 4.8]{Saw1}}: 
%%%%%%%%%%%%%%%%%%%
\begin{lem}\label{Corti}
Let $V$ be a smooth geometrically rational projective surface with $\rho _k(V) =2$ and $(-K_V)^2 \le 4$ over $k$, and admits a Mori conic bundle $\pi :V \to \bP ^1_k$ over $k$. 
Let $\sL$ be a linear pencil on $V$ such that $\Bs (\sL )$ consists of exactly one $k$-rational point $p$. 
Assume that a general member $L$ of $\sL$ satisfies $L \backslash \{ p\} \simeq \bA ^1_k$ and is $\bQ$-linearly equivalent to $a(-K_V) +  bF$ for some $a,b \in \bQ$, where $F$ is the closed fiber of $\pi :V \to \bP ^1_k$ passing through $p$. 
Then $a>0$ and $b<0$. 
\end{lem}
%%%%%%%%%%%%%%%%%%%
\begin{proof}
The assertion $a>0$ can be easily seen by $0 \le (\sL \cdot F) = 2a$ and $0 < (\sL )^2 = a(da+4b)$. 
Suppose $b \ge 0$. 
Let $\Phi _{\sL}:V \dashrightarrow \bP ^1_k$ be the rational map associate to $\sL$, and let $\psi: \bar{V} \to V$ be the shortest succession of blow-ups $p \in \Bs (\sL )$ and its infinitely near points such that the proper transform $\bar{\sL} := \psi ^{-1}_{\ast}(\sL )$ of $\sL$ is free of base points to give rise to a morphism $\bar{\varphi} := \Phi _{\bar{\sL}} \circ \psi$. 
Notice that $\psi$ is defined over $k$ by construction. 
Letting ${\{ \bar{E}_i \}}_{1\le i \le n}$ be the exceptional divisors of $\psi$ with $\bar{E}_n$ the last exceptional one, which is a section of $\bar{\varphi}$, we have: 
\begin{align}\label{en}
(\bar{\sL} \cdot \bar{E}_i) = \left\{ \begin{array}{ll} 0 & \text{if}\ \ 1 \le i \le n-1 \\ 1 & \text{if}\ \ i=n \end{array} \right. 
\end{align}
and
\begin{align}\label{lc}
K_{\bar{V}} - \frac{b}{a} \psi ^{\ast}(F) + \frac{1}{a} \bar{\sL} = \psi ^{\ast} \left( K_V - \frac{b}{a}F + \frac{1}{a} \sL \right) + \sum _{i=1}^nc_i \bar{E}_i
\end{align}
for some rational numbers $c_1,\dots ,c_n$. 
As $a>0$, $b \ge 0$ and $(\bar{\sL})^2 = 0$, we have:
\begin{align*}
-2 &= (\bar{\sL} \cdot K_{\bar{V}}) \\
&= \left( \bar{\sL} \cdot K_{\bar{V}} + \frac{1}{a}\bar{\sL} \right) \\
&\ge \left( \bar{\sL} \cdot K_{\bar{V}} - \frac{b}{a} \psi ^{\ast}(F)   + \frac{1}{a}\bar{\sL} \right) \\
&\underset{(\ref{lc})}{=} \left( \bar{\sL} \cdot \psi ^{\ast} \left( K_V - \frac{b}{a}F + \frac{1}{a} \sL \right) \right) + \sum _{i=1}^nc_i (\bar{\sL} \cdot \bar{E}_i ) \\
&\underset{(\ref{en})}{=} \left( \bar{\sL} \cdot \psi ^{\ast} \left( K_V - \frac{b}{a}F + \frac{1}{a} \sL \right) \right) + c_n .
\end{align*}
Since $K_V - \frac{b}{a}F + \frac{1}{a} \sL \sim _{\bQ} 0$, we have $c_n \le -2$. 
This implies that $(V ,-\frac{b}{a}F+\frac{1}{a}\sL )$ is not log canonical at $p$. 
We will consider whether $F$ is smooth or not in what follows. 

{\it In the case that $F$ is smooth:} 
By {\cite[Theorem 3.1 (1)]{Cor00}}, we have: 
\begin{align} \label{Corti(1)}
i(L_1,L_2;p) > 4\left( 1+ \frac{b}{a} \right) a^2 = 4a(a+b), 
\end{align}
where $L_1$ and $L_2$ are general members of $\sL$. 
Meanwhile, since $L_1$ and $L_2$ meet at only $p$, the left hand side of (\ref{Corti(1)}) can be written as: 
\begin{align*}
i(L_1,L_2;p) &= (\sL )^2 = (-K_V)^2a^2+4ab \le 4a(a+b),
\end{align*}
where we recall that $(-K_V)^2 \le 4$. It is a contradiction to (\ref{Corti(1)}). 

{\it In the case that $F$ is not smooth:} 
We then know that $F_{\kc}$ is a singular fiber of $\pi _{\kc}$, in particular, $F_{\kc}$ consists of exactly two irreducible components $F_1$ and $F_2$ meeting transversely at $p$. 
Hence, $(V ,-\frac{b}{a}F_1-\frac{b}{a}F_2+\frac{1}{a}\sL )$ is not log canonical at $p$. 
By {\cite[Theorem 3.1]{Cor00}}, we have: 
\begin{align} \label{Corti(2)}
i(L_1,L_2;p) > 4\left\{ \left(1+ \frac{b}{a}\right) + \left( 1+ \frac{b}{a}\right)  -1 \right\} a^2 = 4a(a+2b), 
\end{align}
where $L_1$ and $L_2$ are general members of $\sL$. 
By the similar argument as above, we have: 
\begin{align*}
i(L_1,L_2;p) \le 4a(a+b) \le 4a(a+2b), 
\end{align*}
which is a contradiction to (\ref{Corti(2)}). 
\end{proof}
%%%%%%%%%%%%%%%%%%%%%%%%%%%%%%%%%%%%%%%%%%%%%%%%%%%%%%%%%%%%%%%%%%%%%%%%%%%%%%%%%%%%%%%%%%%%%%%%%%%%%%%%%%
\section{Birational maps between geometrically rational projective surfaces}\label{3}
In this section, we recall birational maps between smooth minimal geometrically rational projective surfaces over a perfect field $k$. 
We refer to {\cite{Isk96}}. 

In order to state some results briefly, 
$\fD$ denotes a family of smooth del Pezzo surfaces of Picard rank one defined over $k$ and $\fC$ denotes a family of smooth geometrically rational projective surfaces of Picard rank two with a Mori conic bundle structure over $k$. 
By Lemma \ref{minimal}, any smooth minimal geometrically rational projective surface defined over $k$ is included in $\fD \cup \fC$. 
We then note the following result: 
%%%%%%%%%%%%%%%%%%%
\begin{lem}[{\cite[Theorem 1.6]{Isk96}}]\label{lem(3-1)}
Let $\chi : V \dashrightarrow W$ be a birational map between $V,W \in \fD \cup \fC$. Then: 
\begin{enumerate}
\item If $V \in \fD$ and the number of irreducible components of the $\Gal$-orbit of $x$ is equal to or more than $(-K_V)^2$ for any closed point $x \in V_{\kc}$, then $\chi$ is isomorphic, in other words, $V$ is birational rigid. 
\item If $V \in \fC$ and $(-K_V)^2 \le 0$, then $W \in \fC$ and $(-K_W)^2=(-K_V)^2$. 
\end{enumerate}
\end{lem}
%%%%%%%%%%%%%%%%%%%
Now, we quickly review elementary links between smooth geometrically rational projective surfaces over a perfect field: 
%%%%%%%%%%%%%%%%%%%
\begin{defin}\label{def(3-1)}
Let $V$ and $V'$ be two smooth geometrically rational projective surfaces with $V,V' \in \fD \cup \fC$. 
Then we shall define elementary links of type (I)--(IV) as follows: 
\begin{enumerate}
\item We say that the following commutative diagram is called an {\it elementary link of type (I)}: 
\begin{align*}
\xymatrix{
V \ar[d] & & & V' \ar[lll]_-{\sigma} \ar[d]^-{\pi '} \\
{\Spec (k)} & & & B' \ar[lll]
}
\end{align*}
where $V \in \fD$, $V' \in \fC$, $\pi '$ is a Mori conic bundle, and $\sigma$ is a blowing-up a $\Gal$-orbit of a closed point on $V_{\kc}$. 
\item We say that the following commutative diagram is called an {\it elementary link of type (II)}: 
\begin{align*}
\xymatrix@C=40pt{
%\xymatrix{
V \ar[d]_-{\pi} & Z \ar[l]_-{\sigma} \ar[r]^-{\sigma '} & V' \ar[d]^-{\pi '} \\
B \ar@{=}[rr] & & B' 
}
\end{align*}
where either $V,V' \in \fD$ or $V,V' \in \fC$. 
Moreover, $\sigma$ is a blowing-up a $\Gal$-orbit of a closed point on $V_{\kc}$, $\sigma '$ is a blowing-up a $\Gal$-orbit of a closed point on $V_{\kc}'$, and $\pi$ and $\pi '$ are structure morphisms (resp. Mori conic bundles) if $V,V' \in \fD$ (resp. $V,V' \in \fC$).  
\item We say that the following commutative diagram is called an {\it elementary link of type (III)}: 
\begin{align*}
\xymatrix{
V \ar[d]_-{\pi} \ar[rrr]^-{\sigma '} & & & V' \ar[d] \\
B \ar[rrr] & & & {\Spec (k)} 
}
\end{align*}
where $V \in \fC$, $V' \in \fD$, $\pi$ is a Mori conic bundle, and $\sigma '$ is a blowing-up a $\Gal$-orbit of a closed point on $V_{\kc}'$. 
\item We say that the following commutative diagram is called an {\it elementary link of type (IV)}: 
\begin{align*}
\xymatrix{
V \ar[d]_-{\pi} \ar@{=}[rr] & & V' \ar[d]^-{\pi '}\\
B \ar[r] & {\Spec (k)} & B' \ar[l]
}
\end{align*}
where $V=V' \in \fC$, and $\pi$ and $\pi '$ are different Mori conic bundles. 
\end{enumerate}
\end{defin}
%%%%%%%%%%%%%%%%%%%

%%%%%%%%%%%%%%%%%%%
\begin{lem}[{\cite[Theorem 2.5]{Isk96}}]\label{lem(3-2)}
Any birational map over $k$ between smooth minimal geometrically rational projective surfaces splits into a composition of finitely many elementary links of types (I)--(IV). 
\end{lem}
%%%%%%%%%%%%%%%%%%%
Notice that the classification of elementary links between smooth minimal geometrically rational projective surfaces over a perfect field is known (see {\cite[Theorem 2.6]{Isk96}}, for details). 
In particular, we will use the following result: 
%%%%%%%%%%%%%%%%%%%
\begin{lem}\label{lem(3-3)}
Let $\chi :V \dashrightarrow V'$ be an elementary link. Then: 
\begin{enumerate}
\item If $V \in \fD$ and $(-K_V)^2 \le 3$, then $V' \in \fD$ and $(-K_{V'})^2=(-K_V)^2$. 
\item If $V \in \fD$, $V(k)=\emptyset$ and $(-K_V)^2 \not= 8$, then $V' \in \fD$ and $(-K_{V'})^2=(-K_V)^2$. 
\end{enumerate}
\end{lem}
%%%%%%%%%%%%%%%%%%%
\begin{proof}
See {\cite[Theorem 2.6]{Isk96}}. 
\end{proof}
%%%%%%%%%%%%%%%%%%%
\begin{rem}
For any $V \in \fD$ with $(-K_V)^2=5$, we note that $V$ always has a $k$-rational point ({\cite{Swi72}}). 
\end{rem}
%%%%%%%%%%%%%%%%%%%
The following result will play important role in \S \ref{5}: 
%%%%%%%%%%%%%%%%%%%
\begin{lem}\label{lem(3-4)}
Let $\chi : V \dashrightarrow W$ be a birational map between smooth minimal projective surfaces defined over $k$. 
Then the following four assertions hold: 
\begin{enumerate}
\item Assume that $V \in \fD$ and $(-K_V)^2 \le 4$. 
Then $(-K_W)^2 \le (-K_V)^2$. 
\item Assume that $V \in \fD$ and $V(k) = \emptyset$. 
Then $W \in \fD$, $W(k) = \emptyset$ and $(-K_W)^2 = (-K_V)^2$. 
\item Assume that $V \in \fC$, $V$ is minimal and $(-K_V)^2<8$. 
Then $W \in \fC$ and $(-K_W)^2 = (-K_V)^2$. 
\item Assume that $V \in \fC$ and $V$ is a $k$-form of $\bP ^1_{\kc} \times \bP ^1_{\kc}$ and has no Mori conic bundle structure with a section defined over $k$. 
Then $W$ satisfies all assumptions of $V$. 
%$W \in \fC$ and $W$ has no Mori conic bundle structure with a section defined over $k$. 
\end{enumerate}
\end{lem}
%%%%%%%%%%%%%%%%%%%
\begin{proof}
In (1), we firstly consider the case $(-K_V)^2=4$ and $V(k) \not= \emptyset$. 
Then by Lemma {\cite[Theorem 4.6]{Isk96}} we obtain that $(-K_W)^2=3$ or $4$. 
Hence, this assertion is true. 
The remaining cases follows from Lemmas \ref{lem(3-1)} (1), \ref{lem(3-2)} and \ref{lem(3-3)}. 

In (2), we firstly consider the case $(-K_V)^2=8$. 
By {\cite[Theorem 4.6]{Isk96}}, we then obtain $W \in \fD$, $(-K_W)^2=8$ and $W(k) = \emptyset$, where we note $(-K_W)^2 \not= 6$ by Lemma \ref{degree} because $W$ is $k$-minimal. 
Hence, this assertion is true. 
The remaining cases follows from Lemmas \ref{lem(3-1)} (1), \ref{lem(3-2)} and \ref{lem(3-3)}. 

In (3), since $V$ is minimal, we note that $(-K_V)^2 \in \{ 1,2,4\}$ or $(-K_V)^2 \le 0$ by Lemma \ref{degree}. 
If $(-K_V)^2>0$, then by Lemma {\cite[Theorem 4.9]{Isk96}} we obtain $W \in \fC$ and $(-K_W)^2=(-K_V)^2$. 
Otherwise, it follows from Lemma \ref{lem(3-1)} (2). 

In (4), by {\cite[Theorem 4.9]{Isk96}} $\chi$ splits into a composition of elementary links of types (II) and (IV). 
Hence, it is enough to show the case that $\chi$ is an elementary link of (II) or (IV). 
If $\chi$ is of type (VI), the assertion (4) is clearly true because of $W=V$. 
Hence, in what follows, we may assume that $\chi$ is of type (II). 
In other words, letting $\pi :V \to B$ be a Mori conic bundle, $W \in \fC$ and $\pi ' \circ \varphi :W \to B$ is a Mori conic bundle such that $\pi = \pi ' \circ \chi$. 
Then we further see by {\cite[Theorem 2.6 (ii)]{Isk96}} that $\varphi$ splits into a blowing-up of the $\Gal$-orbit of a closed point $x$ on $V_{\kc}$ and the contraction of proper transforms of fibers including the $\Gal$-orbit of $x$. 
Hence, $\pi '$ admits no section defined over $k$ because so does $\pi$. 
Moreover, $W$ is also $k$-form of $\bP ^1_{\kc} \times \bP ^1_{\kc}$. 
Now, suppose that there exists a Mori conic bundle $\pi '': W \to B''$ admitting a section defined over $k$. 
Letting $\Gamma ''$ be a section of $\pi ''$ defined over $k$, we can write $\Gamma _{\kc} '' \sim F' + a''F''$ for some integer $a''$, where $F'$ and $F''$ are closed fibers of $\pi '$ and $\pi ''$, respectively. 
Since $\pi '$ does not admit any section, $a''$ must be even. 
Hence, we have $F' \in \Pic (W)$ by virtue of $2F'' \in \Pic (W)$. 
This implies that $\pi '_{\kc}$ contains a closed fiber defined over $k$, in particular, this image of this via $\pi '_{\kc}$ is a $k$-rational point on $B_{\kc}$. 
By Lemma \ref{BS}, we thus obtain $B \simeq \bP ^1_k$. 
In other words, $V \simeq \bP ^1_k \times B'''$ for some projective curve $B'''$, so that the second projection $V \simeq \bP ^1_k \times B''' \to B'''$, which is obviously a Mori conic bundle, admits a section defined over $k$. 
This is a contradiction to the hypothesis. 
Therefore, $W$ has no Mori conic bundle structure with a section defined over $k$. 
\end{proof}
%%%%%%%%%%%%%%%%%%%%%%%%%%%%%%%%%%%%%%%%%%%%%%%%%%%%%%%%%%%%%%%%%%%%%%%%%%%%%%%%%%%%%%%%%%%%%%%%%%%%%%%%%%
\section{Proof of Theorem \ref{main(1)}}\label{4}
In this section, we will prove Theorem \ref{main(1)}. 
Let $V$ be a smooth minimal geometrically rational projective surface, whose $-K_V$ is not ample, defined over $k$. 
Since $V$ is not a del Pezzo surface, $\rho _k(V)=2$ and $V$ is endowed with a structure of Mori conic bundle $\pi :V \to B$ defined over $k$. 
%%%%%%%%%%%%%%%%%%%

At first, we shall consider the case $(-K_V)^2=8$. 
%%%%%%%%%%%%%%%%%%%
\begin{lem}\label{lem(4-1)}
With the notation as above, assume further that $(-K_V)^2 =8$. 
Then $V$ contains an $\bA ^1_k$-cylinder. 
\end{lem}
%%%%%%%%%%%%%%%%%%%
\begin{proof}
By the assumption, we notice $V_{\kc} \simeq \bF _m$ for some $m \ge 2$. 
Hence, $V_{\kc}$ contains the minimal section $M$, which is clearly defined over $k$. 
Since $V$ is endowed with only one structure of Mori conic bundle $\pi :V \to B$ over $k$, thus $\pi$ admits the section $M$ defined over $k$. 
As $\pi$ itself is defined over $k$, the base curve $B_{\kc}$ is also equipped with an action of $\Gal$ induced from that on $V_{\kc}$. 
The complement, say $U'$, of a divisor composed of $C$ and the pull-back by $\pi_{\kc}$ of a $\Gal$-orbit on $B_{\kc}$ is then a smooth affine surface defined over $k$. 
The restriction $\varphi := \pi|_{U'}$ of $\pi$ to $U'$ yields a morphism over an affine curve $Z' \subseteq B$. 
By construction, the base extension $\varphi _{\kc}$ is an $\bA^1$-bundle to conclude that so is $\varphi$ by {\cite[Theorem 1]{KM78}}, which implies that there exists an open subset $Z \subseteq Z'$ such that $\varphi^{-1}(Z) \simeq \bA ^1_k \times Z$. 
This completes the proof. 
\end{proof}
%%%%%%%%%%%%%%%%%%%
\begin{lem}\label{lem(4-2)}
With the notation as above, assume further that $(-K_S)^2 =8$. 
Then $V$ contains the affine plane $\bA ^2_k$ if and only if $V(k) \not= \emptyset$. 
\end{lem}
%%%%%%%%%%%%%%%%%%%
\begin{proof}
Assume that $S$ admits a $k$-rational point. 
%Then we know that $S$ is a trivial $k$-form of the Hirzebruch surface $\bF _m$ of degree $m$ for some $m \ge 2$, namely $S \simeq \bP (\sO _{\bP ^1_k} \oplus \sO _{\bP ^1_k}(m))$. 
Let $\pi :S \to B$ a Mori conic bundle. 
Then the base $B$ is a geometrically rational curve admitting a $k$-rational point to conclude that $B$ is isomorphic to $\bP _k^1$. 
Thus, $S$ contains the affine plane $\bA _k^2$. 
The converse direction is obvious. 
\end{proof}
%%%%%%%%%%%%%%%%%%%
In what follows, we shall consider the case $(-K_V)^2\not= 8$. 
By Lemma \ref{degree}, we then notice $(-K_V)^2 \le 4$. 
In order to prove that $V$ contains no $\bA ^1_k$-cylinder, we need the following claim: 
%%%%%%%%%%%%%%%%%%%
\begin{claim}\label{claim}
With the notation as above, assume further that $(-K_V)^2 \le 4$ and $V$ contains an $\bA ^1_k$-cylinder, say $U \simeq \bA ^1_k \times Z$, where $Z$ is a smooth affine curve defined over $k$. 
Notice that the closures in $V$ of fibers of the projection $pr_Z : U \to \bA ^1_k \times Z$ yields a linear system, say $\sL$, on $V$. 
Then $\Bs (\sL )$ consists of exactly one $k$-rational point. 
\end{claim}
%%%%%%%%%%%%%%%%%%%
\begin{proof}
Suppose on the contrary that $\Bs (\sL) = \emptyset$. 
Then the rational map $\Phi _{\sL} :V \dashrightarrow \overline{Z}$ associated with $\sL$ is a morphism, where $\overline{Z}$ is the smooth projective model of $Z$. 
In particular, $\Phi _{\sL}$ is a Mori conic bundle. 
By construction of $\Phi _{\sL}$, we see that $\Phi _{\sL}$ admits a section defined over $k$. 
However, this is a contradiction to Lemma \ref{MCB} (2). 
Thus, we obtain $\Bs (\sL ) \not= \emptyset$. 
Moreover, by the configuration of $\sL$, we obtain this claim. 
\end{proof}
%%%%%%%%%%%%%%%%%%%
By using Claim \ref{claim}, we can show the following lemma: 
%%%%%%%%%%%%%%%%%%%
\begin{lem}\label{lem(4-3)}
With the notation as above, assume further that $(-K_V)^2 \le 4$. 
Then $V$ contains no $\bA ^1_k$-cylinder. 
\end{lem}
%%%%%%%%%%%%%%%%%%%
\begin{proof}
Suppose on the contrary that $V$ contains an $\bA ^1_k$-cylinder $U \simeq \bA ^1_k \times Z$. 
Recall that $V$ is endowed with a structure of a Mori conic bundle $\pi: V \to B$ defined over $k$, where we note $B_{\kc} \simeq \bP ^1_{\kc}$ since $V_{\kc}$ is rational. 
Letting $\sL$ be the same as in Claim \ref{claim}, then $\Bs (\sL )$ consists of exactly one $k$-rational point, say $p$. 
Since $p$ is $k$-rational, so is its image via $\pi$, in particular, $B$ is isomorphic to $\bP_k^1$ by Lemma \ref{BS}. 
Since $Z$ is contained in $\bP ^1_k$ by the similar argument, $\sL$ is a linear pencil on $S$. 
Letting $F$ be the closed fiber of $\pi _{\kc}$ passing through $p$, which is defined over $k$, we can write $\sL \sim _{\bQ} a(-K_V) + bF$ for some $a,b \in \bQ$ since $\Pic (V)_{\bQ}$ is generated by $-K_V$ and $F$. 
In what follows, we consider the two cases that $(-K_V)^2 \le 0$ or $(-K_V)^2 >0$ separately. 

In the case of $(-K_V)^2 \le 0$, by Lemma \ref{Corti} we have $0 < (\sL )^2 = a^2(-K_V)^2+4ab < 0$, which is absurd. 

In the case of $(-K_V)^2 >0$, we notice there exists a curve $C$ on $V$ with $(-K_V \cdot C) \le 0$ since $-K_V$ is not ample. 
Since $a>0$ by Lemma \ref{Corti}, we have $0 \le (\sL \cdot C) = a(-K_V \cdot C) + b(F \cdot C) \le b(F \cdot C)$. 
Hence, $b \ge 0$ by virtue of $(F \cdot C)>0$. 
This is a contradiction to Lemma \ref{Corti}. 
\end{proof}
%%%%%%%%%%%%%%%%%%%
Theorem \ref{main(1)} follows from Lemmas \ref{lem(4-1)}, \ref{lem(4-2)} and \ref{lem(4-3)}. 
%%%%%%%%%%%%%%%%%%%%%%%%%%%%%%%%%%%%%%%%%%%%%%%%%%%%%%%%%%%%%%%%%%%%%%%%%%%%%%%%%%%%%%%%%%%%%%%%%%%%%%%%%%
\section{Proof of Theorem \ref{main(2)}}\label{5}
In this section, we prove Theorem \ref{main(2)}. 
Let $\sigma :S \to V$ be a birational morphism between smooth projective surfaces over $k$ such that $V$ is $k$-minimal. 
Notice that the ``if'' part of Theorem \ref{main(2)} is obvious. 
Hence, we shall prove the ``only if'' part of Theorem \ref{main(2)}, in other words, assume that $S$ contains an $\bA ^1_k$-cylinder in what follows. 
Then we have the following lemma: 
%%%%%%%%%%%%%%%%%%%
\begin{lem}\label{lem(5-1)}
With the notation and the assumption as above, there exists a birational map $\tau :S \dashrightarrow W$ between smooth projective surfaces over $k$ such that $W$ is endowed with a structure of $\bP ^1$-bundle admitting a section defined over $k$. 
\end{lem}
%%%%%%%%%%%%%%%%%%%
\begin{proof}
By the assumption, we can take an $\bA ^1_k$-cylinder in $S$, say $U \simeq \bA ^1_k \times Z$, where $Z$ is a smooth affine curve defined over $k$. 
Then the closures in $S$ of fibers of the projection $pr_Z : U \to \bA ^1_k \times Z$ yields a linear system, say $\sL$, on $S$. 
Let $\Phi _{\sL} :S \dashrightarrow \overline{Z}$ be the rational map associated with $\sL$, where $\overline{Z}$ is the smooth projective model of $Z$. 
Note that either ${\rm Bs}(\sL ) = \emptyset$ or ${\rm Bs}(\sL )$ consists of exactly one $k$-rational point. 
Let $\psi : \bar{S} \to S$ be the shortest succession of blow-ups at a point on $\Bs (\sL )$ and its infinitely near points such that the proper transform $\bar{\sL} := \psi ^{-1}_{\ast}(\sL )$ is free of base points to give rise to a morphism $\bar{\varphi} := \psi \circ \Phi _{\sL}$, where we set $\psi := id$ if ${\rm Bs}(\sL ) = \emptyset$. 
Then $\bar{\varphi}$ is a $\bP ^1$-fibration. 
By contracting $\Gal$-orbits of $(-1)$-curves and subsequently $\Gal$-orbits of (smoothly) contractible curves in the union of all singular fibers of $\bar{\varphi}_{\kc}$, we obtain a birational morphism $\tau' :\bar{S} \to W$ defined over $k$ such that $W$ is endowed with a structure of Mori conic bundle $\pi :W \to \overline{Z}$ satisfying $\bar{\varphi} = \pi \circ \tau'$. 
Since $\bar{\varphi}$ admits a section $\Gamma$ defined over $k$ by construction, the direct image $\tau' _{\ast}(\Gamma)$ is a section of $\pi$ defined over $k$. 
Hence, $\pi$ is a $\bP ^1$-bundle by Lemma \ref{MCB} (2). 
Moreover, the rational map $\tau := \tau' \circ \psi ^{-1}:S \dashrightarrow W$ is birational over $k$. 
\end{proof}
%%%%%%%%%%%%%%%%%%%
\begin{rem}
In the proof of Lemma \ref{lem(5-1)}, assume that $S$ is geometrically rational and ${\rm Bs}(\sL ) \not= \emptyset$. 
Then $\overline{Z}$ is a $k$-form of $\bP ^1_{\kc}$. 
Since $\bar{S}$ contains a $k$-rational point $\bar{p}$ lying on the exceptional divisors of $\psi$ with the last exceptional one, the image $\bar{\varphi} (\bar{p})$ is a $k$-rational point on $\overline{Z}$. 
Hence, we obtain $\overline{Z} \simeq \bP ^1_k$ by Lemma \ref{BS}. 
In particular, we know that $S$ is rational over $k$. 
\end{rem}
%%%%%%%%%%%%%%%%%%%
By Lemma \ref{lem(5-1)}, we know $\kappa (S_{\kc}) = -\infty$. 
Hence, we shall consider the two cases that $S$ is geometrically rational or not separately. 
At first, we treat the case that $S$ is not geometrically rational as follows: 
%%%%%%%%%%%%%%%%%%%
\begin{lem}\label{lem(5-2)}
With the notation and assumptions as above, assume further that $S$ is not geometrically rational. 
Then $V$ contains an $\bA ^1_k$-cylinder. 
\end{lem}
%%%%%%%%%%%%%%%%%%%
\begin{proof}
By Lemma \ref{lem(5-1)}, there exists a birational map $\tau :S \dashrightarrow W$ between smooth projective surfaces over $k$ such that $W$ has a $\bP ^1$-bundle structure $\pi :W \to C$ admitting a section defined over $k$. 
Hence, we obtain a birational map $\chi := \tau \circ \sigma ^{-1}:V \dashrightarrow W$. 
Let $\Gamma$ be a section of $\pi$ defined over $k$. 
Since $V$ is $k$-minimal and is not geometrically rational, $V$ is endowed with only one structure of Mori conic bundle $p :V \to B$. 
In order to show that $\chi _{\ast}(\Gamma )$ is also a section of $p$, we shall take a general fiber of $F$ of $p_{\kc}$. 
Then notice that $\chi ^{-1}_{\ast ,\kc}(F)$ is a fiber of $\pi _{\kc}$ by L\"{u}roth Theorem and $q(V_{\kc})>0$.  
Hence, we have $(\chi _{\ast}(\Gamma )_{\kc} \cdot F) = (\Gamma _{\kc} \cdot \chi ^{\ast}_{\kc}(F)) = (\Gamma _{\kc} \cdot \chi ^{-1}_{\ast ,\kc}(F))=1$ by the projection formula. 
This implies that $\chi _{\ast}(\Gamma )$ is a section of $p$. 
Thus, $V$ contains a vertical $\bA ^1_k$-cylinder with respect to $p$ by the similar argument to Lemma \ref{lem(4-1)}. 
\end{proof}
%%%%%%%%%%%%%%%%%%%
\begin{rem}
Let $V'$ be a smooth minimal projective surface, which is not geometrically rational, defined over $k$ containing an $\bA ^1_k$-cylinder. 
By Lemma \ref{lem(5-2)}, $V'$ is then endowed with a structure of $\bP ^1$-bundle admitting a section defined over $k$. 
Hence, by Lemma \ref{MCB} (2) we know $(-K_{V'})^2=8(1-q(V_{\kc}'))$. 
\end{rem}
%%%%%%%%%%%%%%%%%%%
In what follows, we shall treat the case that $S$ is geometrically rational as follows: 
%%%%%%%%%%%%%%%%%%%
\begin{lem}
With the notation and assumptions as above, assume further that $S$ is geometrically rational. 
Then $V$ contains an $\bA ^1_k$-cylinder. 
\end{lem}
%%%%%%%%%%%%%%%%%%%
\begin{proof}
Suppose on the contrary that $V$ does not contain any $\bA ^1_k$-cylinder. 
By Theorems \ref{DK} and \ref{main(1)}, we then know that $V$ satisfies one of the following properties: 
\begin{itemize}
\item $V$ is of Picard rank one with $(-K_V)^2 \le 4$. 
\item $V$ is of Picard rank one with $V(k) = \emptyset$. 
\item $V$ is of Picard rank two and has no Mori conic bundle structure with a section defined over $k$. 
\end{itemize}
On the other hand, by Lemma \ref{lem(5-1)} we have a birational map $\chi :V \dashrightarrow W$ such that $W$ is endowed with a structure of $\bP ^1$-bundle admitting a section defined over $k$, in particular, $W \in \fC$ and $(-K_W)^2=8$ by Lemma \ref{MCB} (1). 
However, it is a contradiction to Lemma \ref{lem(3-4)}. 
\end{proof}
%%%%%%%%%%%%%%%%%%%
Therefore, this completes the proof of the ``only if'' part of Theorem \ref{main(2)}. 
Namely, we obtain Theorem \ref{main(2)}. 
%%%%%%%%%%%%%%%%%%%%%%%%%%%%%%%%%%%%%%%%%%%%%%%%%%%%%%%%%%%%%%%%%%%%%%%%%%%%%%%%%%%%%%%%%%%%%%%%%%%%%%%%%%
\section{Remarks on Corollary \ref{cor}}\label{6}
%%%%%%%%%%%%%%%%%%%
\subsection{Images of cylinders via birational maps}
%%%%%%%%%%%%%%%%%%%
\subsubsection{}
Let $\theta :X_1 \dashrightarrow X_2$ be a birational map between smooth projective varieties such that $X_2$ contains an $\bA ^1_k$-cylinder. 
Notice that $X_1$ obviously contains an $\bA ^1_k$-cylinder provided that $\theta$ is a morphism. 
However, it is subtle in general if $\theta$ is not a morphism. 
As an example, let $X$ be a smooth Fano threefold of Picard rank one, of Fano index one and of genus seven. 
It is known that $X$ is rational, hence, we obtain a birational map $\theta :X \dashrightarrow \bP ^3_k$. 
However, it is expected that $X$ does not contain any cylinder (see {\cite[Conjecture 3.13]{CPPZ21}}). 
Meanwhile, Corollary \ref{cor} implies that $X_1$ always contains an $\bA ^1_k$-cylinder provided that $X_1$ and $X_2$ are two-dimensional. 
%%%%%%%%%%%%%%%%%%%
\subsubsection{}
By using Lemma \ref{vertical} and Corollary \ref{cor}, we obtain the following corollary: 
%%%%%%%%%%%%%%%%%%%
\begin{cor}\label{cor(2)}
Let $f_1: X_1 \to Y_1$, $f_2:X_2 \to Y_2$ and $g:Y_1 \to Y_2$ be projective dominant morphisms over a field $K$ of characteristic zero such that general fibers of $g \circ f_1$ and $g$ are smooth surfaces, and assume that there exists a birational map $\theta :X_1 \dashrightarrow X_2$ such that $g \circ f_1 = f_2 \circ \theta$ (see the following commutative diagram): 
\begin{align*}
\xymatrix@C=40pt{
X_1 \ar@{-->}[rr]^-{\theta} \ar[d]_-{f_1} & & X_2 \ar[d]^-{f_2}\\
Y_1 \ar[rr]_{g} & & Y_2 
}
\end{align*}
If $f_1$ admits a vertical $\bA ^1_K$-cylinder, then so does $f_2$. 
\end{cor}  
%%%%%%%%%%%%%%%%%
\begin{proof}
By the definition of a vertical cylinder (see Definition \ref{vertical:def}) and the assumption, $g \circ f_1$ admits a vertical $\bA ^1_K$-cylinder. 
Thus, we obtain this assertion by Lemma \ref{vertical} and Corollary \ref{cor}. 
\end{proof}
%%%%%%%%%%%%%%%%%
%Corollary \ref{cor(2)} particularly implies that any elementary link of type III (by means of {\cite[(3.4) Definition]{Cor95}}) keeps the existence or absence of vertical $\bA ^1_K$-cylinders. 
%However, the author does not know whether elementary links of others types satisfy or not the similar property. 
%%%%%%%%%%%%%%%%%%%
\subsection{Application to cylinders in singular surfaces}
Let $S$ be a normal projective surface defined over $k$. 
If $S$ is smooth and does not contain any $\bA ^1_k$-cylinder, then so does any smooth projective surface $S'$ over $k$, which is birational to $S$ over $k$, by Corollary \ref{cor}. 
Now, we shall consider the case without assuming that $S$ is smooth. 
Then we have: 
%%%%%%%%%%%%%%%%%%%
\begin{cor}\label{cor(3)}
With the notation as above, assume further that the minimal resolution $\widetilde{S}$ of $S$ over $k$ contains no $\bA ^1_k$-cylinder. 
Then so does any normal projective surface $S'$ over $k$, which is birational to $S$ over $k$. 
\end{cor}
%%%%%%%%%%%%%%%%%%%
\begin{proof}
We shall take any normal projective surface $S'$ over $k$ admitting a birational map $\chi :S' \dashrightarrow S$ over $k$. 
Let $\widetilde{S}'$ be the minimal resolution of $S'$ over $k$. 
Then there exists a birational map $\widetilde{\chi}: \widetilde{S}' \dashrightarrow \widetilde{S}$ over $k$ such that the following diagram is commutative: 
\begin{align*}
\xymatrix@C=40pt{
\widetilde{S}' \ar[d]_{\sigma '} \ar@{-->}[rr]^-{\widetilde{\chi}} & & \widetilde{S} \ar[d]^{\sigma} \\
S' \ar@{-->}[rr]_-{\chi} & & S
}
\end{align*}
Here, in the above diagram, $\sigma :\widetilde{S} \to S$ and $\sigma ' :\widetilde{S}' \to S'$ are the minimal resolutions $S$ and $S'$ over $k$, respectively. 
By the assumption and Corollary \ref{cor}, we see that $\widetilde{S}'$ contains no $\bA ^1_k$-cylinder. 
Namely, so does $S'$. This completes the proof. 
\end{proof}
%%%%%%%%%%%%%%%%%%%
Recall that the existing condition of $\bA ^1$-cylinders in normal del Pezzo surfaces of Picard rank one with only Du Val singularities is completely determined ({\cite{Saw2}}). 
By using Corollary \ref{cor(3)}, we can construct many examples of singular del Pezzo surfaces containing no $\bA ^1$-cylinder over algebraically non-closed fields (Picard rank of these surfaces need not necessarily be one). 
%%%%%%%%%%%%%%%%%%%
\begin{eg}
Let $S \subseteq \bP ^4_k$ be a singular intersection of two quadrics defined over $k$ such that $S_{\kc}$ has exactly two Du Val singular points of type $A_1$ lying in the same $\Gal$-orbit and there is no line on $S_{\kc}$ joining these singularities.
Note that $S$ is an Iskovskih surface (see {\cite[p.\ 74]{CT88}}, for this definition). 
Let $\sigma :\widetilde{S} \to S$ be the minimal resolution over $k$. 
Then we know that $\widetilde{S}$ is $k$-minimal and satisfies $(-K_{\widetilde{S}})^2=4$ (see also the weighted dual graph in {\cite[Type 3 of Proposition 6.1]{CT88}}). 
Hence, $\widetilde{S}$ contains never $\bA ^1_k$-cylinder by Theorem \ref{main(1)}. 
Thus, by Corollary \ref{cor(3)} every normal projective surface over $k$, which is birational to $S$ over $k$, does not contain any $\bA ^1_k$-cylinder. 
\end{eg}
%%%%%%%%%%%%%%%%%%%
Note that $S$ does not always contain an $\bA ^1_k$-cylinder even if its minimal resolution $\widetilde{S}$ contains an $\bA ^1_k$-cylinder. 
In fact, we can construct the following example: 
%%%%%%%%%%%%%%%%%%%
\begin{eg}
Let $S$ be a normal del Pezzo surface defined over $k$ and let $\sigma :\widetilde{S} \to S$ be the minimal resolution over $k$. 
Assume that $S_{\kc}$ has exactly two Du Val singular points of type $D_4$, and $\rho _k(\widetilde{S})=9$ (if $k=\kc$, this assumption naturally holds). 
Since $\widetilde{S}$ is then rational over $k$, we know that $\widetilde{S}$ contains an $\bA ^1_k$-cylinder by Theorem \ref{main(2)}. 
Meanwhile, since $S$ is a normal del Pezzo surface of the Picard rank one and of degree $1$ with only two Du Val singularities of type $D_4$, $S$ does not contain any $\bA ^1_k$-cylinder by {\cite{Saw2}}. 
Note that this fact also follow from {\cite{CPW16b}} because of $\rho _{\kc}(S_{\kc})=1$. 
\end{eg}
%%%%%%%%%%%%%%%%%%%%%%%%%%%%%%%%%%%%%%%%%%%%%%%%%%%%%%%%%%%%%%%%%%%%%%%%%%%%%%%%%%%%%%%%%%%%%%%%%%%%%%%%%%

\end{document}